\documentclass[11pt,a4paper,margin=1in]{amsart}
\usepackage{url}
\usepackage{amsmath, amsthm, amssymb}
\usepackage{mathrsfs}
\usepackage{fancyhdr}
\usepackage{hyperref}
\usepackage{lineno}
\usepackage{geometry} 
\usepackage{comment}

\theoremstyle{plain}

\newtheorem{theorem}{Theorem}[section]
\newtheorem{lemma}[theorem]{Lemma}
\newtheorem{proposition}[theorem]{Proposition}

\theoremstyle{definition}

\newtheorem{definition}[theorem]{Definition}
\newtheorem{example}[theorem]{Example}

\theoremstyle{remark}

\newtheorem{remark}[theorem]{Remark}
\newtheorem{question}[theorem]{Question}

\begin{document}
\title[Co-Cyclic and Sub-Cyclic Mappings]{Common Fixed Points of Co-Cyclic and Sub-Cyclic Mappings in Complete Metric Spaces}
\author{$^1$Dev Raj Joshi,  $^{2, \ast}$Sanjay Mishra,  $^3$Piyush Kumar Tripathi}
\maketitle
\begin{center}
{
\footnotesize $^1$Department of Mathematics, Far-Westen University, Tikapur Multiple Campus Tikapur, Kailali, Nepal.\\ $^{2, 3}$ Department of Mathematics, Amity School of Applied Sciences, Amity University Lucknow Campus, Uttar Pradesh, Lucknow, India \\

$^{\ast}$Corresponding Author: drsmishraresearch@gmail.com}
\end{center}

\begin{abstract}
In this paper, we introduce two new classes of cyclic-type mappings, namely,
co-cyclic mappings and sub-cyclic mappings, defined on the union of two
subsets of a metric space. Under suitable Jungck-type contractive conditions
and compatibility assumptions, we establish unique common fixed point
theorems for these mappings in complete metric spaces. The proofs are based
on the construction of appropriate sequences, the associated contractive
conditions, and the completeness of the underlying metric space. Examples
are provided to illustrate the proposed concepts and to demonstrate the
applicability of the main results. Finally, some open questions and possible
directions for further research concerning co-cyclic and sub-cyclic mappings
are presented.
\end{abstract}



\section{Introduction and Preliminaries}
\label{sec:Introduction and Preliminaries}

Fixed point theory is an important area of nonlinear analysis with
applications in several branches of mathematics and related disciplines.
Among the various directions in fixed point theory, the study of cyclic
mappings has received considerable attention due to its usefulness in
extending classical fixed point results to mappings defined on unions of
subsets of a metric space. In particular, cyclic conditions provide a
natural framework for studying fixed points and best proximity points when
the domain is represented as the union of two or more subsets. In this
section, we recall some fundamental concepts and known results concerning
cyclic mappings, cyclic representations, semi-cyclic mappings, and
contractive conditions. These preliminary notions provide the necessary
background and motivation for the introduction of the new classes of
co-cyclic and sub-cyclic mappings considered in this paper.

Throughout this paper, $M$ denotes a non-empty set, and $\mu$ denotes a
metric on $M$. The concept of a cyclic mapping was first introduced by
Kirk et al.~\cite{W.12} as follows.

\begin{definition}\label{def:cyclic-mapping}
Let $V$ and $W$ be two non-empty subsets of a metric space $(M,\mu)$. 
A mapping $f: V\cup W \to V\cup W$ is called cyclic if $f(V)\subseteq W$ and $f(W)\subseteq V$. 
\end{definition}

A natural question arises: Do there exist two or more self-mappings satisfying the conditions
\[
f(V)\subseteq g(V)\subseteq W
\quad\text{and}\quad
f(W)\subseteq g(W)\subseteq V?
\]
If such mappings exist, how can we establish a common fixed point theorem for them? 
In this paper, we aim to answer these questions and investigate related concepts. 
We also provide examples to illustrate and verify the proposed concepts and results.

The following fixed point theorem for cyclic mappings was proved by Kirk et al.~\cite{W.12}.

\begin{theorem}\label{thm:kirk-cyclic}
Let $V$ and $W$ be two non-empty closed subsets of a complete metric space
$(M,\mu)$. Let $f:V\cup W\to V\cup W$ be a cyclic mapping such that there
exists $q\in(0,1)$ satisfying
\[
\mu(fx,fy)\leq q\mu(x,y),
\qquad \forall x\in V,\ \forall y\in W.
\]
Then $f$ has a unique fixed point in $V\cap W$.
\end{theorem}

Following the work of Kirk et al.~\cite{W.12}, many authors introduced and
studied different types of cyclic mappings, such as semi-cyclic mappings,
cyclic representations of sets, and generalized cyclic representations.
Several fixed point results have been established for different types of
cyclic contractions. In $2011$, De Ia Sen~\cite{M.Dla} introduced a general
contractive condition for cyclic mappings. In the same year, Petric
~\cite{M.A} proved best proximity point theorems for weak cyclic Kannan
contractions. Karapinar and Erhan~\cite{Erthan} introduced best proximity
point results for different types of cyclic contractions. In $2012$, Abbas
et al.~\cite{M.Ab.}, Aydi et al.~\cite{H.A.C}, Karapinar et al.~\cite{E.K},
Karapinar~\cite{E.}, and Karapinar and Nashine~\cite{E.H} studied different
types of cyclic contractions and established corresponding fixed point
theorems. Subsequently, many researchers obtained fixed point results for
cyclic mappings under various contractive conditions; see, for example,
\cite{A.H,G.H,H.M.D,G.N,H.M,J.M,J03,J04,J05,P.S}.

The study of common fixed points for cyclic-type mappings can be viewed
broadly through two approaches: cyclic representations and semi-cyclic
mappings. The cyclic representation of two mappings and the corresponding
common fixed point problem were initiated by Shatanawi and Postolache
~\cite{W.S} in $2013$. Subsequently, in $2015$, Abbas et al.~\cite{M.09}
introduced the following definition of a cyclic representation for two
mappings on several subsets, together with a generalized contractive condition:

\begin{definition}\cite{M.09}\label{def:generalized-cyclic-contraction}
Let $\{M_i:i=1,2,\ldots,k\}$ be a finite collection of non-empty closed
subsets of a metric space $(M,\mu)$, and let
\[
f,g:\bigcup_{i=1}^{k}M_i\to\bigcup_{i=1}^{k}M_i
\]
be two mappings. The pair $(f,g)$ is called a generalized cyclic
contraction if the following conditions are satisfied:
\begin{enumerate}
    \item $\displaystyle\bigcup_{i=1}^{k}M_i$ has an $(f,g)$-cyclic
    representation with respect to the collection
    $\{M_i:i=1,2,\ldots,k\}$;

    \item There exist two control functions $\psi$ and $\phi$ such that
    \[
    \psi\bigl(\mu(f(x),f(y))\bigr)
    \leq
    \psi\bigl(M_{g,f}(x,y)\bigr)
    -
    \phi\bigl(M_{g,f}(x,y)\bigr),
    \]
    where
    \[
    M_{g,f}(x,y)
    =
    \max\left\{
    \mu(g(x),g(y)),
    \mu(g(x),f(x)),
    \mu(g(y),f(x))
    \right\},
    \]
    for any $x\in M_i$ and $y\in M_{i+1}$, where
    $i=1,2,\ldots,k$ and $M_{k+1}=M_1$.
\end{enumerate}
\end{definition}

Yamaod et al.~\cite{O.Y} also used the cyclic representation of two
mappings for two subsets of a metric space in $2015$. Subsequently, many authors established common fixed point results by using cyclic representations of two mappings in different types of metric spaces; see, for example, \cite{A.R,J.L,J.L.M,M.B,W.G}. Another important direction in fixed point theory is the study of common fixed points for semi-cyclic mappings. H. Aydi et al.~\cite{H.A} introduced the concept of semi-cyclic mappings as follows:

\begin{definition}\label{def:semi-cyclic-mapping}
Let $V$ and $W$ be two non-empty subsets of a metric space $(M,\mu)$.
The mappings $f,g:V\cup W\to V\cup W$ are called \emph{semi-cyclic} if $f(V)\subseteq W$ and $g(W)\subseteq V$. 
\end{definition}

Many researchers have established common fixed point theorems for
semi-cyclic mappings under various types of contractive conditions; see, for example, \cite{M.L,S.W,S.W.X,S.W.Y}. Based on the above literature, we observe that the study of common fixed points for cyclic mappings in different metric spaces has mainly developed through two approaches:
cyclic representations of sets with respect to two self-mappings and
semi-cyclic mappings. In this work, we introduce co-cyclic and sub-cyclic
mappings on the union of two closed subsets of a metric space and establish
a common fixed point result for such mappings by using a Jungck-type
contractive condition~\cite{G.03}.

Gerald Jungck~\cite{G.03} introduced the following type of contraction.

\begin{definition}\label{def:g-contraction}
Let $f,g:M\to M$ be two self-mappings of a complete metric space
$(M,\mu)$ such that $f(M)\subseteq g(M)$. If there exists a constant
$q\in(0,1)$ such that
\[
\mu(f(x),f(y))\leq q\,\mu(g(x),g(y)),
\qquad \forall x,y\in M,
\]
then $f$ is called a $g$-contraction on the metric space $(M,\mu)$.
\end{definition}

Further, Gerald Jungck~\cite{G.04} introduced the concept of compatible
mappings as follows.

\begin{definition}\label{def:compatible-mappings}
Two self-mappings $f$ and $g$ of a metric space $(M,\mu)$ are called
compatible if
\[
\lim_{n\to\infty}\mu\bigl(f(g(x_n)),g(f(x_n))\bigr)=0
\]
whenever $\{x_n\}$ is a sequence in $M$ such that
\[
\lim_{n\to\infty}f(x_n)
=
\lim_{n\to\infty}g(x_n).
\]
\end{definition}
\section{Main Results}\label{sec:Main Results}

In this section, we present the main contributions of the paper. Motivated by
the existing theories of cyclic representations and semi-cyclic mappings, we
introduce two new classes of mappings, namely, co-cyclic mappings and
sub-cyclic mappings. These concepts describe different cyclic relationships
between two mappings defined on the union of two subsets of a metric space.

We then establish common fixed point theorems for these mappings under
suitable Jungck-type contractive conditions and compatibility assumptions.
The proofs are based on the construction of appropriate iterative sequences
and the completeness of the underlying metric space. The obtained results
extend the framework of cyclic mappings to a setting involving two mappings
and provide sufficient conditions for the existence and uniqueness of their
common fixed points. Illustrative examples are also provided to demonstrate
the applicability of the proposed concepts and results.

To establish the main results, we present some new concepts regarding the convergence of two equivalent sequences, along with examples in support of these concepts.

\begin{definition}\label{def:equivalent-sequences}
Two sequences $\{x_n\}$ and $\{y_n\}$ in a metric space $(M,\mu)$ are
called equivalent if
\[
\lim_{n\to\infty}\mu(x_n,y_n)=0.
\]
\end{definition}

\begin{lemma}\label{lem:equivalent-cauchy-sequences}
Let $(M,\mu)$ be a complete metric space. If $\{x_n\}$ and $\{y_n\}$ are
Cauchy sequences in $M$ such that
\[
\lim_{n\to\infty}\mu(x_n,y_n)=0,
\]
then $\{x_n\}$ and $\{y_n\}$ converge to the same point in $M$.
\end{lemma}

\begin{definition}\label{def:non-cauchy-equivalent}
Let $\{x_n\}$ and $\{y_n\}$ be two sequences in a metric space $(M,\mu)$.
If
\[
\lim_{n\to\infty}\mu(x_n,y_n)=0
\]
and neither $\{x_n\}$ nor $\{y_n\}$ is a Cauchy sequence, then
$\{x_n\}$ and $\{y_n\}$ are called non-Cauchy equivalent sequences.
\end{definition}

\begin{example}\label{exmp:equivalent-cauchy}
Let $x_n=2-\frac{1}{n}$ and $y_n=2+\frac{1}{n}$, and let $\mu(x,y)=|x-y|$ be the usual metric on $\mathbb{R}$. Then
$\{x_n\}$ and $\{y_n\}$ are Cauchy sequences, and
\[
\lim_{n\to\infty}\mu(x_n,y_n)
=
\lim_{n\to\infty}|x_n-y_n|
=
\lim_{n\to\infty}\frac{2}{n}
=0.
\]
Thus, $\{x_n\}$ and $\{y_n\}$ are equivalent sequences and both converge
to the same limit point $2$.
\end{example}

\begin{example}\label{exmp:non-cauchy-equivalent} 
Let $x_n=n^2$ and $y_n=n^2+\frac{1}{n}$, and let $\mu(x,y)=|x-y|$ be the usual metric on $\mathbb{R}$. Both
$\{x_n\}$ and $\{y_n\}$ are not Cauchy sequences. Moreover,
\[
\lim_{n\to\infty}\mu(x_n,y_n)
=
\lim_{n\to\infty}|x_n-y_n|
=
\lim_{n\to\infty}\frac{1}{n}
=0.
\]
Therefore, $\{x_n\}$ and $\{y_n\}$ are non-Cauchy equivalent sequences.
\end{example}

The main purpose of this paper is to introduce the following two types of cyclic mappings with respect to two subsets of a metric space $(M,\mu)$.

\begin{definition}\label{def:co-cyclic-mapping}
Let $V$ and $W$ be two non-empty subsets of a metric space $(M,\mu)$.
Two self-mappings $f,g:V\cup W\to V\cup W$ are called co-cyclic if
\[
f(V)\subseteq g(V)\subseteq W
\quad\text{and}\quad
f(W)\subseteq g(W)\subseteq V.
\]
\end{definition}

\begin{definition}\label{def:sub-cyclic-mapping}
Let $V$ and $W$ be two non-empty subsets of a metric space $(M,\mu)$.
Let $f,g:V\cup W\to V\cup W$ be two self-mappings such that
\[
f(V)\subseteq g(W)\subseteq W
\quad\text{and}\quad
f(W)\subseteq g(V)\subseteq V.
\]
Then $f$ is called sub-cyclic in $g$.
\end{definition}

Now we prove common fixed point theorems on the basis of Co- cyclic and Sub-cyclic mappings and supporting example for the main result.

\begin{proposition}\label{pro:intersection-invariance}
Let $V$ and $W$ be two non-empty subsets of a metric space $(M,\mu)$.
Let $f,g:V\cup W\to V\cup W$ be co-cyclic mappings satisfying
\[
f(V)\subseteq g(V)\subseteq W
\quad\text{and}\quad
f(W)\subseteq g(W)\subseteq V.
\]
Then
\[
f(V\cap W)\subseteq V\cap W
\quad\text{and}\quad
g(V\cap W)\subseteq V\cap W.
\]
Equivalently, for every $x\in V\cap W$,
\[
f(x)\in V\cap W
\quad\text{and}\quad
g(x)\in V\cap W.
\]
\end{proposition}

\begin{proof}
Let $x\in V\cap W$. Then $x\in V$ and $x\in W$. Since $f(V)\subseteq W$ and $f(W)\subseteq V$, 
we have $f(x)\in W$ and $f(x)\in V$. Hence, $f(x)\in V\cap W$. Therefore, $f(V\cap W)\subseteq V\cap W$. 

Similarly, since $g(V)\subseteq W$ and $g(W)\subseteq V$, we obtain $g(x)\in W$ and $g(x)\in V$. Thus, $g(x)\in V\cap W$. Therefore, $g(V\cap W)\subseteq V\cap W$
\end{proof}

\begin{theorem}\label{thm:co-cyclic-common-fixed-point}
Let $V$ and $W$ be two non-empty closed and convex subsets of a complete
metric space $(M,\mu)$ such that $V\cap W\neq\varnothing$. Let $f,g:V\cup W\to V\cup W$ be two continuous co-cyclic mappings satisfying
\[
f(V)\subseteq g(V)\subseteq W
\quad\text{and}\quad
f(W)\subseteq g(W)\subseteq V.
\]
Suppose that there exists $q\in(0,1)$ such that
\[
\mu(f(x),f(y))
\leq
q\,\mu(g(x),g(y)),
\qquad \forall x\in V,\ \forall y\in W.
\]
If $f$ and $g$ are compatible, then $f$ and $g$ have a unique common
fixed point in $V\cap W$.
\end{theorem}

\begin{proof}
Let $U=V\cap W$. Since $V\cap W\neq\varnothing$ and $V$ and $W$ are closed, $U$ is a
non-empty closed subset of $M$. Hence $U$ is a complete metric space.

We first observe that $f$ and $g$ are self-mappings on $U$. Indeed, let
$x\in U$. Then $x\in V$ and $x\in W$. Since $f(V)\subseteq W$ and $f(W)\subseteq V$, we obtain $f(x)\in W$ and $f(x)\in V$. Thus, $f(x)\in U$. 
Similarly, since $g(V)\subseteq W$ and $g(W)\subseteq V$, we have $g(x)\in U$. Therefore, $f(U)\subseteq U$ and $g(U)\subseteq U$. 

Moreover, by the co-cyclic property, $f(U)\subseteq g(U)$. Choose an arbitrary point $t_0\in U$. Since $f(t_0)\in f(U)\subseteq
g(U)$, there exists $t_1\in U$ such that $f(t_0)=g(t_1)$. Inductively, we can construct a sequence $\{t_n\}$ in $U$ satisfying
\[
f(t_n)=g(t_{n+1}),\qquad n\geq0.
\]

Define $p_n=g(t_n)$, $n\geq0$. Then
\[
p_{n+1}=g(t_{n+1})=f(t_n).
\]
Since $t_n,t_{n+1}\in U\subseteq V\cap W$, the contractive condition
yields
\[
\begin{aligned}
\mu(p_{n+1},p_{n+2})
&=\mu(g(t_{n+1}),g(t_{n+2}))\\
&=\mu(f(t_n),f(t_{n+1}))\\
&\leq q\,\mu(g(t_n),g(t_{n+1}))\\
&=q\,\mu(p_n,p_{n+1}).
\end{aligned}
\]
Therefore, by induction,
\[
\mu(p_n,p_{n+1})
\leq
q^n\mu(p_0,p_1),
\qquad n\geq0.
\]

We now show that $\{p_n\}$ is a Cauchy sequence. Let $m>n$. By the
triangle inequality,
\[
\begin{aligned}
\mu(p_n,p_m)
&\leq
\sum_{j=n}^{m-1}\mu(p_j,p_{j+1})\\
&\leq
\sum_{j=n}^{m-1}q^j\mu(p_0,p_1)\\
&\leq
\frac{q^n}{1-q}\mu(p_0,p_1).
\end{aligned}
\]
Since $q\in(0,1)$, the right-hand side tends to zero as $n\to\infty$.
Hence $\{p_n\}$ is a Cauchy sequence in $U$.

Since $U$ is complete, there exists $z\in U$ such that $p_n=g(t_n)\longrightarrow z$. Also,
\[
f(t_n)=g(t_{n+1})=p_{n+1}\longrightarrow z.
\]
Thus,
\[
\lim_{n\to\infty}f(t_n)
=
\lim_{n\to\infty}g(t_n)
=
z.
\]

Since $f$ and $g$ are compatible,
\[
\lim_{n\to\infty}
\mu\bigl(fg(t_n),gf(t_n)\bigr)=0.
\]
By the continuity of $f$ and $g$, we have
\[
fg(t_n)=f(g(t_n))\longrightarrow f(z)
\]
and
\[
gf(t_n)=g(f(t_n))\longrightarrow g(z).
\]
Therefore, $\mu(f(z),g(z))=0$, 
and hence $f(z)=g(z)$. 

We now prove that $z$ is a common fixed point of $f$ and $g$. Since
$z,t_n\in U\subseteq V\cap W$, the contractive condition gives
\[
\mu(f(z),f(t_n))
\leq
q\,\mu(g(z),g(t_n)).
\]
Taking the limit as $n\to\infty$, we obtain
\[
\mu(f(z),z)
\leq
q\,\mu(g(z),z).
\]
Since $f(z)=g(z)$, it follows that
\[
\mu(g(z),z)
\leq
q\,\mu(g(z),z).
\]
Hence $(1-q)\mu(g(z),z)\leq0$. Since $q\in(0,1)$, we obtain $\mu(g(z),z)=0$, 
and therefore $g(z)=z$. 
Since $f(z)=g(z)$, we also have $f(z)=z$. Thus $z$ is a common fixed point of $f$ and $g$ in $V\cap W$.

Finally, suppose that $u,v\in V\cap W$ are two common fixed points of
$f$ and $g$. Then
\[
f(u)=g(u)=u
\quad\text{and}\quad
f(v)=g(v)=v.
\]
Using the contractive condition, we obtain
\[
\mu(u,v)
=
\mu(f(u),f(v))
\leq
q\,\mu(g(u),g(v))
=
q\,\mu(u,v).
\]
Thus, $(1-q)\mu(u,v)\leq0$. 
Since $1-q>0$, we have $\mu(u,v)=0$. 
and hence $u=v$. Therefore, the common fixed point is unique.
\end{proof}

Next, suppose that $f$ is a cyclic mapping and $g$ is not necessarily
cyclic, but $f(V\cup W)\subseteq g(V\cup W)$. Under these assumptions, we establish a common fixed point theorem.

\begin{theorem}\label{thm:sub-cyclic-common-fixed-point}
Let $V$ and $W$ be two non-empty closed subsets of a complete metric
space $(M,\mu)$. Let $f,g:V\cup W\to V\cup W$ be two continuous self-mappings such that:

\begin{enumerate}
    \item[(i)] $f$ is sub-cyclic in $g$, that is, $ f(V)\subseteq g(W)\subseteq W$ and $ f(W)\subseteq g(V)\subseteq V$;

    \item[(ii)] there exists $q\in(0,1)$ such that
    \[
    \mu(f(x),f(y))
    \leq
    q\,\mu(g(x),g(y)),
    \qquad \forall x\in V,\ \forall y\in W.
    \]
\end{enumerate}
If $f$ and $g$ are compatible, then they have a unique common fixed
point in $V\cap W$.
\end{theorem}

\begin{proof}
Choose an arbitrary point $x_0\in V$. Since $f(V)\subseteq g(W)$, 
there exists $x_1\in W$ such that $f(x_0)=g(x_1)$. Again, since $f(W)\subseteq g(V)$, 
there exists $x_2\in V$ such that $f(x_1)=g(x_2)$. 

Continuing inductively, we obtain a sequence $\{x_n\}$ in $V\cup W$
such that
\[
x_{2n}\in V,\qquad x_{2n+1}\in W,
\]
and
\[
f(x_n)=g(x_{n+1}),
\qquad n\geq0.
\]

Since $x_n$ and $x_{n+1}$ belong alternately to $V$ and $W$, the
contractive condition gives
\[
\begin{aligned}
\mu(g(x_{n+1}),g(x_{n+2}))
&=
\mu(f(x_n),f(x_{n+1}))\\
&\leq
q\,\mu(g(x_n),g(x_{n+1})).
\end{aligned}
\]
Therefore, by induction,
\[
\mu(g(x_n),g(x_{n+1}))
\leq
q^n\mu(g(x_0),g(x_1)),
\qquad n\geq0.
\]

We now show that $\{g(x_n)\}$ is a Cauchy sequence. Let $m>n$. By the
triangle inequality,
\[
\begin{aligned}
\mu(g(x_n),g(x_m))
&\leq
\sum_{k=n}^{m-1}
\mu(g(x_k),g(x_{k+1}))\\
&\leq
\sum_{k=n}^{m-1}
q^k\mu(g(x_0),g(x_1))\\
&\leq
\frac{q^n}{1-q}\mu(g(x_0),g(x_1)).
\end{aligned}
\]
Since $q\in(0,1)$, the right-hand side tends to zero as $n\to\infty$.
Thus $\{g(x_n)\}$ is a Cauchy sequence.

Since $(M,\mu)$ is complete, there exists $z\in M$ such that $g(x_n)\longrightarrow z$. Moreover,
\[
f(x_n)=g(x_{n+1})\longrightarrow z.
\]
Hence
\[
f(x_n)\longrightarrow z
\quad\text{and}\quad
g(x_n)\longrightarrow z.
\]

Now observe that
\[
g(x_{2n})=f(x_{2n-1})\in f(W)\subseteq V
\]
and
\[
g(x_{2n+1})=f(x_{2n})\in f(V)\subseteq W.
\]
Thus the even subsequence $\{g(x_{2n})\}$ lies in $V$, while the odd
subsequence $\{g(x_{2n+1})\}$ lies in $W$. Since both subsequences
converge to $z$ and $V$ and $W$ are closed, we obtain $z\in V\cap W$. 

Since $f$ and $g$ are compatible and
\[
f(x_n)\longrightarrow z
\quad\text{and}\quad
g(x_n)\longrightarrow z,
\]
we have
\[
\mu(fg(x_n),gf(x_n))\longrightarrow0.
\]
By the continuity of $f$ and $g$,
\[
fg(x_n)=f(g(x_n))\longrightarrow f(z)
\]
and
\[
gf(x_n)=g(f(x_n))\longrightarrow g(z).
\]
Therefore, $\mu(f(z),g(z))=0$, 
and hence $f(z)=g(z)$. 

We now show that $z$ is a common fixed point. Since $z\in V\cap W$, we
may apply the contractive condition to $z$ and $x_n$ for each $n$ by
choosing the appropriate order of the two points. Thus,
\[
\mu(f(z),f(x_n))
\leq
q\,\mu(g(z),g(x_n)).
\]
Taking the limit as $n\to\infty$, we obtain
\[
\mu(f(z),z)
\leq
q\,\mu(g(z),z).
\]
Since $f(z)=g(z)$, it follows that
\[
\mu(g(z),z)
\leq
q\,\mu(g(z),z).
\]
Therefore, $(1-q)\mu(g(z),z)\leq0$. 
Since $q\in(0,1)$, we obtain $g(z)=z$. 
Consequently, $f(z)=g(z)=z$. Hence $z$ is a common fixed point of $f$ and $g$.

Finally, suppose that $u,v\in V\cap W$ are two common fixed points of
$f$ and $g$. Then
\[
f(u)=g(u)=u
\quad\text{and}\quad
f(v)=g(v)=v.
\]
Applying the contractive condition, we obtain
\[
\mu(u,v)
=
\mu(f(u),f(v))
\leq
q\,\mu(g(u),g(v))
=
q\,\mu(u,v).
\]
Thus, $(1-q)\mu(u,v)\leq0$. 
Since $1-q>0$, we obtain $\mu(u,v)=0$, 
and hence $u=v$. Therefore, the common fixed point is unique.
\end{proof}

\begin{example}\label{exmp:sub-cyclic}
Let $M=\mathbb{R}$ be endowed with the usual metric $\mu(x,y)=|x-y|$. Let $V=[-1,0]$ and $W=[0,1]$. Then $V$ and $W$ are non-empty closed and convex subsets of the complete
metric space $(M,\mu)$. 
Define two mappings $f,g:V\cup W\to V\cup W$ by
\[
f(x)=-\frac{x}{4}
\qquad\text{and}\qquad
g(x)=-\frac{x}{2}.
\]
For $x\in V$, we have
\[
f(x)\in\left[0,\frac14\right]
\quad\text{and}\quad
g(x)\in\left[0,\frac12\right],
\]
and hence $f(V)\subseteq g(V)\subseteq W$. 

Similarly, for $x\in W$, we have
\[
f(x)\in\left[-\frac14,0\right]
\quad\text{and}\quad
g(x)\in\left[-\frac12,0\right],
\]
and hence $f(W)\subseteq g(W)\subseteq V$. Therefore, $f$ and $g$ are co-cyclic mappings.

Moreover, for every $x,y\in V\cup W$,
\[
\begin{aligned}
\mu(f(x),f(y))
&=
\left|-\frac{x}{4}+\frac{y}{4}\right|\\
&=
\frac14|x-y|\\
&=
\frac12\left(\frac12|x-y|\right)\\
&=
\frac12\mu(g(x),g(y)).
\end{aligned}
\]

Thus, the contractive condition holds with $q=\frac12\in(0,1)$. Also, $f(g(x))=\frac{x}{8}$ and $g(f(x))=\frac{x}{8}$, for every $x\in V\cup W$. Hence $f$ and $g$ commute and, consequently,
they are compatible.

Since $V\cap W=\{0\}$ and $f(0)=g(0)=0$, it follows that $0$ is the unique common fixed point of $f$ and $g$ in
$V\cap W$.
\end{example}

\begin{example}\label{exmp:sub-cyclic-example}
Let $M=\mathbb{R}$ be endowed with the usual metric $\mu(x,y)=|x-y|$. Then $(M,\mu)$ is a complete metric space. Let $V=[-1,0]$ and $W=[0,1]$. Define two mappings $f,g:V\cup W\to V\cup W$ by $f(x)=-\frac{x}{4}$ and $g(x)=\frac{x}{2}$.

For $x\in V$, we have
\[
f(V)=\left[0,\frac14\right]
\quad\text{and}\quad
g(W)=\left[0,\frac12\right].
\]
Therefore, $f(V)\subseteq g(W)\subseteq W$. Similarly,
\[
f(W)=\left[-\frac14,0\right]
\quad\text{and}\quad
g(V)=\left[-\frac12,0\right],
\]
and hence $f(W)\subseteq g(V)\subseteq V$. Thus, $f$ is sub-cyclic in $g$.

Moreover, for every $x\in V$ and $y\in W$,
\[
\begin{aligned}
\mu(f(x),f(y))
&=
\left|-\frac{x}{4}+\frac{y}{4}\right| 
\frac14|x-y|,
\end{aligned}
\]
whereas
\[
\begin{aligned}
\mu(g(x),g(y))
&=
\left|\frac{x}{2}-\frac{y}{2}\right| =
\frac12|x-y|.
\end{aligned}
\]
Consequently,
\[
\mu(f(x),f(y))
=
\frac12\mu(g(x),g(y)).
\]
Hence the contractive condition is satisfied with $q=\frac12\in(0,1)$. 

Furthermore, $f(g(x)) = -\frac{x}{8} $ and $g(f(x)) = -\frac{x}{8}$, 
for every $x\in V\cup W$. Thus, $f$ and $g$ commute and hence are
compatible.

Finally, $V\cap W=\{0\}$ and $f(0)=g(0)=0$. Therefore, by Theorem~\textup{(3.2)}, $f$ and $g$ have the unique common
fixed point $0$ in $V\cap W$.
\end{example}

\begin{remark}\label{rk:generalization}\hfill
\begin{enumerate}
    \item If $g$ is the identity mapping in
    Theorem~\ref{thm:sub-cyclic-common-fixed-point}, then the
    sub-cyclic condition reduces to the usual cyclic condition $f(V)\subseteq W$ and $ f(W)\subseteq V$.    Moreover, the contractive condition reduces to
    \[
    \mu(f(x),f(y))\leq q\,\mu(x,y),
    \qquad x\in V,\quad y\in W.
    \]
    Hence, Theorem~\ref{thm:sub-cyclic-common-fixed-point} reduces to
    Theorem~\ref{thm:kirk-cyclic}.

    \item Consequently,
    Theorem~\ref{thm:sub-cyclic-common-fixed-point} extends the cyclic
    fixed point theorem of Kirk et al.~\cite{W.12}.
\end{enumerate}
\end{remark}

\section{Conclusion}\label{sec:Conclusion}

In this paper, we introduced two new classes of mappings, namely,
co-cyclic mappings and sub-cyclic mappings, on the union of two subsets of a
metric space. We established unique common fixed point results for these
mappings under suitable Jungck-type contractive conditions and compatibility
assumptions. The proofs are based on the construction of appropriate
sequences and the completeness of the underlying metric space.

The proposed concepts extend the study of common fixed points for cyclic-type
mappings, and the examples illustrate the applicability of the main results.
As a future scope, the proposed co-cyclic and sub-cyclic mappings may be
studied under other contractive conditions and extended to more general
metric-type spaces. Further research may also focus on best proximity point
results and common fixed point theorems for generalized versions of these
mappings.

\section{Open Questions}
\label{sec:Open Questions}

The introduction of co-cyclic and sub-cyclic mappings gives rise to several
natural questions for further investigation.

\begin{question}
Can the common fixed point results obtained in this paper be extended
to co-cyclic and sub-cyclic mappings satisfying other types of
contractive conditions, such as Kannan-type, Chatterjea-type, or
generalized contractive conditions?
\end{question}

\begin{question}
Can common fixed point results for co-cyclic and sub-cyclic mappings
be established without assuming the continuity or compatibility of
the mappings? If so, what alternative conditions can replace these
assumptions?
\end{question}

\begin{question}
Can the notions of co-cyclic and sub-cyclic mappings be extended to a
finite collection of subsets of a metric space, analogous to cyclic
representations involving more than two subsets?
\end{question}

\begin{question}
Can best proximity point results be established for co-cyclic and
sub-cyclic mappings when the underlying subsets have an empty
intersection?
\end{question}

\begin{question}
Do the common fixed point results for co-cyclic and sub-cyclic
mappings remain valid in more general spaces, such as $b$-metric
spaces, partial metric spaces, or other metric-type spaces?
\end{question}

\end{document}